\documentclass[11pt]{amsart}
\usepackage[top=3.5cm, bottom=3.5cm, left=2.5cm, right=2.5cm]{geometry}
\title{The Conjugacy Ratio of  abelian-by-cyclic groups}
\author{David Guo}

\usepackage[T1]{fontenc}
\usepackage{lmodern}
\usepackage{textcomp}

\usepackage{graphicx}
\usepackage{framed}
\usepackage{amsmath}

\usepackage{amssymb}
\usepackage{amsfonts}
\usepackage{mathrsfs}

\usepackage{array}

\usepackage{enumerate}

\usepackage{amsthm}  
\usepackage{mathtools}
\usepackage{physics}
\usepackage{xcolor}
\usepackage[normalem]{ulem}

\usepackage{soul}
\newcommand\numberthis{\addtocounter{equation}{1}\tag{\theequation}}

\usepackage{dutchcal}

\usepackage[capitalise,nameinlink ,noabbrev]{cleveref}

\providecommand{\N}{\mathbb{N}}
\providecommand{\Z}{\mathbb{Z}}
\providecommand{\Q}{\mathbb{Q}}

\providecommand{\rt}{\rtimes}

\providecommand{\GE}{\mathcal{G}}

\renewcommand{\L}{\mathcal{L}}
\providecommand{\R}{\mathcal{R}}

\providecommand{\br}{\mathcal{r}}
\providecommand{\bl}{\mathcal{l}}
\providecommand{\id}{\mathrm{id}}
\renewcommand{\O}{\mathcal{O}}
\providecommand{\littleo}{\mathcal{o}}

\providecommand{\N}{\mathbb{N}}
\providecommand{\Z}{\mathbb{Z}}
\providecommand{\Q}{\mathbb{Q}}
\providecommand{\B}{\mathbb{B}}

\renewcommand{\a}{\alpha}

\renewcommand{\b}{\beta}

\providecommand{\iso}{ \cong}

\renewcommand{\l}{ \lambda}

\providecommand{\e}{ \epsilon}
\providecommand{\s}{ \psi}
\renewcommand{\d}{ \delta}

\providecommand{\p}{ \phi}
\renewcommand{\P}{ \mathcal{P}}

\providecommand{\subg}{\leqslant}

\renewcommand{\i}{^{-1}}
 
\providecommand{\d}[2]{\diam(G/H)} 
 
\providecommand{\nsubg}{\trianglelefteq}

\providecommand{\rt}{\rtimes} 
\providecommand{\subeq}{\subseteq}

\providecommand{\Pi}{P_{i}}

\providecommand{\cpc}[1]{ \operatorname{ Cycpc}  {(#1) }} 
\renewcommand{\c}[1]{ \operatorname{ cyc}  {(#1) }} 
\DeclareMathOperator{\Tor}{Tor}

\DeclareMathOperator{\Aut}{Aut}

\DeclareMathOperator{\im}{Im}
\providecommand{\glnq}{GL(n, \Q)}
\providecommand{\glnz}{GL(n, \Z)}

\theoremstyle{plain} 
\newtheorem*{thm*}{Theorem}
\newtheorem{thm}{Theorem}[section] 
\newtheorem{lem}[thm]{Lemma} 
 
\newtheorem{cor}[thm]{Corollary} 
 
\newtheorem*{notation}{Notation}

\theoremstyle{definition}

\theoremstyle{remark} 
\newtheorem*{rem}{Remark}

\usepackage[numbers]{natbib}
\usepackage{cite}

\providecommand{\dist}[2]{\Delta^{#1}_{#2}(r)}

\begin{document}

\begin{abstract}
 Let $G = K \rt  \langle t \rangle $  be a finitely generated group where $K$ is abelian and $\langle t\rangle$ is the infinite cyclic group. Let $ R  $ be a finite symmetric subset of $K$ such that $S = \{  (r,1),(0,t^{\pm 1}) \mid r \in R \}$ is a generating set of $G$.  We prove that the spherical conjugacy ratio, and hence the conjugacy ratio, of $G$  with respect to $S$ is $0$ unless $G$ is virtually abelian, confirming a conjecture of Ciobanu, Cox and Martino in this case.  We also show that the Baumslag--Solitar group $\mathrm{BS}(1,k)$, $k\geq 2$,  has a one-sided  Følner sequence $F$ such that the conjugacy ratio with respect to $F$ is non-zero, even though $\mathrm{BS}(1,k)$ is not virtually abelian.  This is in contrast to two-sided Følner sequences,  where Tointon showed that the conjugacy ratio with respect to a two-sided Følner sequence is positive if and only if the group is virtually abelian.
 \end{abstract}
\maketitle

\tableofcontents

\section{Introduction}

Let $F$ be a finite group.  The \textit{degree of commutativity} of $F$, denoted $\mathrm{dc} (F)$, is the probability that two randomly chosen elements commute, i.e.  
\[
\operatorname{dc}(F):=\frac{\left|\left\{(a, b) \in F^2: a b=b a\right\}\right|}{|F|^2}.
\] 
Intuitively, the closer $\operatorname{dc}(F)$ is to $1$ the “more abelian” $F$ will be. Indeed, Neumann \citep[Theorem 1]{Neumann}  showed that the degree of commutativity of the finite group tells us about the structure of the group in the following sense:
\begin{thm} \label{Neumann}
   Let $F$ be a finite group such that $\mathrm{dc}(F) \geq \alpha>0$. Then $F$ has a normal subgroup $\Gamma$ of index at most $\alpha^{-1}+1$ and a normal subgroup $H$ of cardinality at most $\exp \left(O\left(\alpha^{-O(1)}\right)\right)$ such that $H \subset \Gamma$ and $\Gamma / H$ is abelian.  
\end{thm}

Given a subset $A$ of a group, we will write $c(A)$ to be the number of distinct conjugacy classes meeting $A$. It is well-known and easy to check that the degree of commutativity is equal to the proportion of the number of conjugacy classes to the number of elements in a finite group $F$, i.e. 
\begin{equation} \label{degree_ratio}
    \operatorname{dc}(F)=\frac{ c(F)}{|F|}.
\end{equation}
In particular, if the \emph{conjugacy ratio} $c(F)/|F|$ is at least $\alpha$, then $F$ satisfies the conclusion of Neumann's theorem.

Recently, there have been efforts made to generalise both the quantities in \eqref{degree_ratio} to infinite groups. Throughout the rest of this paper, we will assume that $G$ is a group generated by a finite symmetric subset $S$  containing the identity. Antolín, Martino and Ventura \citep{AMV} defined  the \textit{degree of commutativity} of $G$ with respect to a sequence $(\mu_n)$ of finitely supported measures on $G$ by
\[
\operatorname{dc}(G)=\limsup_{n\to\infty}\,(\mu_n\times\mu_n)(\left\{(a, b) \in G \times G: a b=b a\right\}).
\]
In the special case where $\mu_r$ is the uniform probability measure on $S^r=\{s_1\cdots s_r:s_i\in S\}$, this 
becomes the \textit{degree of commutativity with respect to $S$},
\begin{equation} \label{deg_inf}
\operatorname{dc}_S(G) =  \limsup _{r \rightarrow \infty} \frac{\left|\left\{(a, b) \in S^r \times S^r: a b=b a\right\}\right|}{\left|S^r\right|^2}.
\end{equation}
They showed that if $G$ is residually finite and  the generating set $S$ satisfies 
\begin{equation} \label{ratio_of_consecutive_balls}
  \frac{\left|S^{r+1}\right|}{\left|S^r\right|} \rightarrow 1  
\end{equation}
as $r \rightarrow \infty$,
then $\operatorname{dc}_S(G)>0$ if and only if $G$ is virtually abelian. Groups with polynomial growth satisfy \eqref{ratio_of_consecutive_balls} by  \citep{Tessera2007}; it could be that it is only the groups of polynomial growth that satisfy \eqref{ratio_of_consecutive_balls}.

Tointon \citep[Theorem 1.9 (3)]{commuting_tointon} later gave a general condition on a sequence $(\mu_n)$ of measures sufficient to imply such a result. In particular, this applies if $\mu_n$ is uniform on a Følner sequence, or the distribution of the $n$-th step of a simple random walk. As a special case, he removed the hypothesis that $G$ is residually finite from Antolín, Martino and Ventura's result. He also obtained a more quantitative conclusion as in \cref{Neumann}.

 In a similar manner to \eqref{deg_inf}, Cox \citep{Cox} introduces the \emph{conjugacy ratio} $\operatorname{cr}_S(G)$ of $G$ with respect to $S$ via
\begin{equation} \label{conjugay_ratio}
    \operatorname{cr}_S(G)=\limsup _{r \rightarrow \infty} \frac{ c(S^r)
}{\left|S^r\right|}.
\end{equation}

Cox asked whether the analogue of \eqref{degree_ratio} would hold for finitely generated group \citep[Question 2]{Cox}, i.e. whether $\operatorname{\mathrm{dc}}_S(G) = \operatorname{cr}_S(G)$; he also asked whether $ \operatorname{cr}_S(G) > 0$ if and only if $G$ is virtually abelian \citep[Question 1]{Cox}. As remarked above, both of these statements hold for finite groups. 
In his later paper with Ciobanu and Martino, they conjectured the second question explicitly \citep[Conjecture 1.1]{conjugacy_ratio_Laura_Charles} and confirmed this conjecture for some classes of groups.  They first verified their conjecture for any residually finite group $G$ and any finite symmetric generating set $S$ containing the identity and satisfying \eqref{ratio_of_consecutive_balls}. This result was also strengthened by Tointon, who considered the conjugacy ratio with respect to Følner sequences \citep{commuting_tointon}.  { In general, a sequence $\left(F_n\right)_{n=1}^{\infty}$ of finite subsets of $G$ is said to be a {\textit{left-Følner sequence}}, or simply a \textit{Følner sequence}, if for every $x \in G$,
\[
\frac{\left|x F_n  \triangle F_n\right|}{\left|F_n\right|} \rightarrow 0;
\]
similarly, $\left(F_n\right)_{n=1}^{\infty}$ is said to be a  \textit{right-Følner sequence} if for every $x \in G$,
\begin{equation} \label{right_folner}
  \frac{\left| F_n  x\triangle F_n\right|}{\left|F_n\right|} \rightarrow 0.  
\end{equation}

We call $\left(F_n\right)_{n=1}^{\infty}$ a \textit{two-sided Følner sequence} if it is both a left- and right-Følner sequence.  It is well known that if  $S$ satisfies \eqref{ratio_of_consecutive_balls}, then $\left(S^r\right)_{r=1}^{\infty}$ is a two-sided Følner sequence.
  }
  Tointon defined the conjugacy ratio of $G$ with respect to a Følner sequence $F=\left(F_n\right)_{n=1}^{\infty}$ of finite subsets of $G$ via
$$
\operatorname{cr}_F(G)=\limsup _{n \rightarrow \infty} \frac{  c(F_n)  }{\left|F_n\right|}.
$$
He proved that the conjugacy ratio with respect to any two-sided Følner sequence is equal to the degree of commutativity with respect to the same Følner sequence, and in particular positive if and only if the group is virtually abelian \citep[Corollary 8.2 and Proposition 8.5]{commuting_tointon}. Again, { {as a special case of his result}, {Tointon removed from Ciobanu, Cox and Martino's result the hypothesis that $G$ needs to be residually finite}}.

In \citep[\S 4]{conjugacy_ratio_Laura_Charles}, Ciobanu, Cox and Martino's also verify this conjecture for several important classes of groups of exponential growth, including: hyperbolic groups with respect to any generating set; the lamplighter group with respect to the standard generating set; and right-angled Artin groups with respect to the generating set associated with its defining graph. 

The aim of this paper is to study the conjugacy ratio of the abelian-by-cyclic groups. 
Let $G = K \rt_\p  \langle t \rangle $  be a finitely generated group where $K$ is abelian, $\langle t\rangle$ is the infinite cyclic group and the semidirect product is defined by  $\p \in \Aut(K)$. To simplify the notation, we often identify $(\id,t^n) \in G$ as $t^n$, and  $(k, \id) \in G$ as $k \in K$. We will prove Cox's conjecture for finitely generated abelian-by-cyclic groups with respect to certain generating sets.  This includes the lamplighter groups (for a more general class of generating sets than is given by Ciobanu, Cox and Martino), and more generally the wreath product of an abelian group with $\Z$, as well as the Baumslag-Solitar group $BS(1, k)$ for $k \geq 2$. 

 {We will make use of the big-$\O$-notation throughout this paper: let $f$ and $g$ be two real-valued functions defined on some unbounded subset of real numbers. We write  $f(x)=\O(g(x))$ if there exists a positive real number $M$ and a real number $x_0$ such that \[
|f(x)| \leq M|g(x)| \quad \text { for all } x \geq x_0.
\]}
For conciseness of notation, we will sometimes write $f(x) \ll g(x)$ for $f(x)=\O(g(x))$, and  $f(x) \asymp g(x)$ to indicate that both $f(x)=\O(g(x))$ and $g(x)=\O(f(x))$.

\begin{thm} \label{main_thm}
Let $G = K \rt_\p  \langle t \rangle $  be a finitely generated abelian-by-cyclic group with exponential growth.  Let $ R  $ be a finite symmetric subset of $K$ such that $S = \{  (r,1),(0,t^{\pm 1}) \mid r \in R \}$ is a generating set of $G$. {Then $$\frac{c(S^r) }{\left|S^r\right|} = \O\left(\frac{\log r}{r}\right).$$}
\end{thm}

According to the Milnor-Wolf Theorem on the growth of solvable groups, an abelian-by-cyclic group has either polynomial or exponential growth. Ciobanu, Cox and Martino have verified their conjugacy ratio conjecture for all finitely generated groups with polynomial growth with respect to any generating set. Therefore, it remains to check \cref{main_thm} for the case when the group has exponential growth. 
It follows from \citep{conjugacy_growth_for_solvable_groups_Breuillard} that every abelian-by-cyclic group with exponential growth has exponential conjugacy growth. In fact, the exponential growth rates are equal in some groups, for example, the lamplighter group with respect to the standard generating set (see \citep{Growth_series_lamplighter} and \citep{Mercier}). Therefore, if we want to compute a sequence that converges to zero and is a uniform upper bound sequence for the sequences defined in \eqref{conjugay_ratio} for all abelian-by-cyclic groups with exponential growth, we should expect the convergence to be slow in some sense. \cref{main_thm} is more quantitative than saying that such groups have zero conjugacy ratio; it also gives a bound on how quickly the conjugacy ratio decays to its limit.

\begin{rem}De las Heras, Klopsch and Zozaya \citep[Theorem A]{degree_of_commutativity_of_wreath_products} have shown that if $G$ is a wreath product of a finitely generated group $H$ with $\Z$ then $\operatorname{\mathrm{dc}}_S(G)=0$ for any finite generating set $S$. In particular, if $H$ is abelian and $S$ is of the form described in \cref{main_thm} then $\operatorname{\mathrm{dc}}_S(G) = \operatorname{cr}_S(G)$, giving a positive answer to Cox's Question 2.
\end{rem}

 {
By combining \cref{main_thm} with previous research on conjugacy ratios, we can deduce the following:
\begin{cor}\label{main_thm_cor}
    Let $G = K \rt_\p  \langle t \rangle $  be a finitely generated abelian-by-cyclic group.  Let $ R  $ be a finite symmetric subset of $K$ such that $S = \{  (r,1),(0,t^{\pm 1}) \mid r \in R \}$ is a generating set of $G$. Then $\lim _{r \rightarrow \infty} \frac{c(S^r) }{\left|S^r\right|} = 0$ if and only if $G$ is not virtually abelian.
\end{cor} 
    
}

 {As mentioned before, \cref{main_thm} can be applied to several important classes of soluble groups.  

\begin{cor}
 Let $G$ be a wreath product of a finitely generated abelian group $H$ with $\Z$. Then $\operatorname{cr}_S(G)=0$ for the standard generating set $S$.
\end{cor}

\begin{cor}
 Let $G$ be the soluble Baumslag-Solitar group $BS(1, k)$ with $k \geq 2$. Then $\operatorname{cr}_S(G)=0$ for the standard generating set $S$.
\end{cor}}

The generating sets of the form in \cref{main_thm} give a nice description of the geodesics of the group, as it helps us examine the number of distinct elements obtained via the cyclic permutation of a geodesic. In general, it appears to be quite tricky to study the geodesics and conjugacy growth with respect to other generating sets, see, for example, \citep{bregman}, \citep{collin} and \citep{PUTMAN2006190}.  

We will also show that the assumption in Tointon's result on conjugacy ratios \citep[Corollary 8.2]{commuting_tointon} is optimal in the sense that we cannot remove the requirement that the Følner sequence needs to be two-sided. We thank Romain Tessera for suggesting that such a result should hold.

\begin{thm} \label{main_thm_2}
    There exists a right  Følner sequence of the soluble Baumslag-Solitar group $BS(1, k)$, $k\geq 2$, such that the conjugacy ratio of $BS(1, k)$ with respect to this right  Følner sequence is $1$.
\end{thm}

The structure of this paper is as follows: In Section 2, we recall the geodesics and the presentation of abelian-by-$\Z$ groups of the form  $G = K \rt_\p  \langle t \rangle $. In Section 3, we prove that the subgroup of periodic points of $\p$ in $K$ is polynomially distorted in  $G$,  the key ingredient of the proof for \cref{main_thm}, which we will present in Section 4. Lastly, in Section 5, we will present the proof of \cref{main_thm_2}.

\subsection*{Acknowledgements} 
I am very grateful to Matthew Tointon for suggesting the problem and for his ongoing support. I  also thank Corentin Bodart, Yves Cornulier, Besfort Shala, Atticus Stonestrom and Romain Tessera for helpful conversations.  I am grateful to Charles Cox for discussions about recent work related to the conjugacy ratio of groups. I also want to thank Laura Ciobanu Charles Cox, and Alain Valette for their comments on a draft of this paper, as well as an anonymous referee for a careful reading of the paper and numerous helpful suggestions.

\section{Abelian-by-$\Z$ groups }
Given a finite symmetric set $S$, we will denote $W(S)$ the set of all words in the alphabet $S$. Given $w \in W(S)$, we denote $|w|$ as the length of the reduced form of $w$ (i.e. $|w|$ is the length of the freely reduced word obtained by iteratively deleting the subwords of the form  $s s^{-1}$ in $w$). With a slight abuse of language, we will say two elements in $W(S) \cup G$ are equal (resp. conjugate) if they are equal (resp. conjugate) as group elements.  Also, we will write $\cpc{w} $ to be the set of words that can be obtained via cyclic permutation on $w$.

Now, let $G$ be a group generated by a finite symmetric set  $S$. With a slight abuse of notation,  the \emph{length} of an element $g \in G$, denoted by $|g|$, is the length of a shortest word $w$ in $S$ that represents $g$, i.e. $|g|=\min \{|w| \mid w \in  W(S), w= g \}$. In this case, we say $w$ is a \emph{geodesic}. We denote $S^r$ (or $B_S(r)$  when the group is abelian) the ball of radius $r$  in the Cayley graph of $G$ relative to S, i.e. $S^r$ contains the set of elements with length at most $r$.
Similarly, we let  $S(r)$ denote the sphere of radius $r$, i.e. $S(r)$ contains the set of elements with length exactly $r$.  We will write $g \sim h$ to denote that $g$ and $h$ are conjugate, and write $[g]_G$ for the conjugacy class of $g$. 
We say that a word $w$ is a conjugacy geodesic for $[g]_G$ if it is a geodesic and represents an element of the shortest length in $[g]_G$.

\subsection{Geodesics and conjugacy geodesics in abelian-by-$\Z$ groups.} \label{description_of_the_generating_sets}

Let $G = K \rt_\p  \langle t \rangle $  be a finitely generated abelian-by-cyclic group.  {Note that if  $ \{(k_1, t^{p_1}), \ldots , (k_n, t^{p_n}) \mid k_i \in K, p_i \in \Z  \}$ is a finite generating set for $G$, then it is easy to see that $ \{(k_1, 1), \ldots, (k_n, 1), (0, t) \}$ is also finite generating set for $G$; 
with $R =  \{(\pm k_1, 1), \ldots ,  ( \pm k_n, 1), (0,1)  \}$, we have that $S =  R \cup \{t^{\pm 1}\} $ is a finite symmetric generating set for $G$. We will fix $S$ for the remainder of this section.}  {We will generalise results of Parry \citep[Section 3]{Parry}  and Choi, Ho, and Pengitore \citep[Section 2]{odd_trace}, providing a description of geodesics in $G$ with respect to the generating set $S$.} Let $g=\left(x, t^m\right) \in G$. We will call $m$ the \emph{$t$-exponent} sum of $g$.  It is clear that we can express $g$ as a word in $W(S)$ in the following form:
\begin{equation} \label{g_expression}
g= t^{m_0} u_1 t^{m_1} u_2 t^{m_2} \ldots u_{h-1} t^{m_{h-1}} u_h t^{m_h}
\end{equation}
 where $h$ is a non-negative integer, $m_0,  m_h \in\{-1,0,1\}, m_1, \ldots, m_h \in\{-1,1\}$ and $u_1, \ldots, u_h \in W(R)$.   We obtain $|g|$ by minimizing
\begin{equation} \label{g_length}
\sum_{i=0}^h\left|m_i\right|+\sum_{i=1}^h\left|u_i\right|.
\end{equation}
Suppose that $m \geqslant 0$.  We can rewrite line \eqref{g_expression} as

\begin{equation} \label{g_rewrite}
g=\left(t^{m_0} u_1 t^{-m_0}\right)\left(t^{m_0+m_1} u_2 t^{-m_0-m_1}\right) \ldots\left(t^{m_0+\cdots+m_{k-1}} u_h t^{-m_0-\cdots-m_{h}}\right) t^{m_0+\cdots+m_h}.
\end{equation}

Note that every term in parentheses lies in $K$. Therefore, they commute and their product lies in $K$. Moreover $m_0+\cdots+m_h=m$. We do not increase line \eqref{g_length} by combining the term that starts with the same $t$-exponent in line \eqref{g_rewrite} and ordering them according to the partial sums $m_0+\cdots+m_i$. Let $q=\max \left\{0, m_0, m_0+m_1,  \ldots, m\right\}  $ and $ p=\max \left\{0,-m_0,-m_0-m_1, \ldots,-m\right\} $. Without increasing line \eqref{g_length} we may rewrite line \eqref{g_rewrite} as in the following form:
$$g  =\left(t^{-p} y_{-p} t^p\right)\left(t^{-p+1} y_{-p+1} t^{p-1}\right)\left(t^{-p+2} y_{-p+2} t^{p-2}\right) \ldots\left(t^q y_q t^{-q}\right) t^m \\
 =\left(\sum_{i=-p}^q t^{-i} y_{-i} t^i, t^m\right)$$
for some  elements $y_{-p}, \ldots, y_q \in W(R)$. We just proved the following:

\begin{lem}   \label{contains_all_elements}
The set of words { $$ W' = \{ w= t^{-p} u_0 t u_1 t   \ldots u_{l-1} t u_{d}  t^{-q}  t^m \mid p, q, d, m \geq 0, \quad d=p+q, \quad  u_i \in W(R), \quad   w \text{ is a geodesic} \}$$ } contains all the geodesic words in $W(S)$ whose t-exponent sum is non-negative, i.e.  given $g = (x,t^m)  \in G$ with  $m \geq 0$, there exists  $w(g) \in W'$, such that $w(g)$ is a geodesic and represents $g$.

\end{lem}

Next, we want to find a subset of $W'$ that contains a complete set of representatives of conjugacy classes for elements with non-negative $t$-exponent sum.  We will adapt the idea from \citep [Proposition 17]{conjugacy_growth_of_the_bs_groups} where  Ciobanu,  Evetts, and  Ho studied the conjugacy geodesics in Baumslag--Solitar group $\mathrm{BS}(1,k)$. We will write $$  \mathcal{C}_0 = \{  u_0 t u_1 t  u_2 t  \ldots u_{d-1} t u_{d}  t^{-d} \in W' \mid d \geq 0, u_i \in W(R) \}.$$ For $m > 0$, we define $$\mathcal{C}_m = \{  u_0 t u_1 t  u_2 t \ldots u_{m-1} t \in W' \mid u_i \in W(R) \}.$$ Lastly, we denote $\mathcal{C}_+ = \cup _ {m>0} \mathcal{C}_m$.

\begin{lem} \label{contain_all_conjuacy_geodesic}

 Suppose $g$ is a minimal length conjugacy representative.

\begin{enumerate} [(i)]

\item If the $t$-exponent sum of $g$ is $0$, then there exists $w \in \mathcal{C}_0$ such that $  g \in \cpc{w}$.

\item Similarly, if the $t$-exponent sum of $g$ is $m>0$, then there exists $w \in \mathcal{C}_m $ such that $  g \in \cpc{w}$.

\end{enumerate}

\end{lem}

\begin{proof}

(i) is straightforward. 
For (ii), suppose that $w$ contains $t^{-1}$ as a reduced word. We will find a word conjugate to $w$ with a shorter length. By Lemma \ref{contains_all_elements}, after cyclically permuting $w$ if necessary, we may assume that 
$w = u_0 t u_1 t  u_2 t  \ldots u_{l-1} t u_{l} t^{m-l}$.  
Since $w$ contains $t^{-1}$ non-trivially, we have $m-l < 0$. Clearly, $u_l$ can not be the empty word.  It follows that
\begin{align*}
w &\sim   t u_{l} t^{m-l}u_0 t u_1 t  u_2 t   \ldots  t u_{l-m} t   \ldots t u_{l-1} &(\text{via cyclic permutation})  \\
&=  t u_{l} (t^{m-l}u_0 t u_1 t  u_2 t   \ldots  t) u_{l-m} t   \ldots t u_{l-1}   \\
& =  t^{m-l+1}u_0 t u_1 t  u_2 t  \ldots  t (u_{l} u_{l-m})  \ldots u_{l-1}  &(\text {since } t^{m-l}u_0 t u_1 t  u_2 t   \ldots  t \in K).
\end{align*}
The final word has a shorter length compared to $w$ as we removed $2$ symbols. Therefore, if $w$ is a conjugacy geodesic, $w$ cannot contain $t^{-1}$ non-trivially.

\end{proof}

This paper will involve many counting arguments. For simplicity, we will make use of these  two standard combinatorial notations: given $n \in \N$, we denote ${[n]}$ for $\{1, \ldots, n\}$ and ${[n]_0}$ for $\{0, \ldots, n\}$. Next, let $f(r)$ be a non-negative function that converges to $\infty$.  We will denote  $U_f (r)$ to be the set of elements in $S^r $ that can be expressed as words in the set  $W'$  (defined in \cref{contains_all_elements})  where $t$ appears at most $  f (r)$  times.

 {\begin{lem} [Most  elements on $S^r$ contains many $t$'s] \label{most_elts_have_many_t}
 Let $f(r)$ be a non-negative function that converges to $\infty$. Then  $\left|U_f (r)\right|= r^{\mathcal{O}(f(r))}$.
\end{lem}
}

\begin{proof}
Every $w \in W'$ of length at most $r$ can be written as
\[
w = t^{-p} u_0 t u_1 t \cdots u_{l-1} t u_d t^{-q} t^m
\]
for some integers $p, q, d \geq 0$ with $d = p + q$, $m \in [r]_0$, and each $u_i$ representing an element in $B_R(r)$. Suppose that there are at most $f(r)$ occurrences of the letter $t$ in $w$, so in particular $d \in [f(r)]$.

Since $|B_R(r)| = \mathcal{O}\bigl(r^{|R|}\bigr)$, we can deduce that there are at most
\[
(f(r))^2 \, r \, |B_R(r)|^{f(r)} = \left(\mathcal{O}\bigl(r^{|R|}\bigr)\right)^{f(r)} = r^{\mathcal{O}(f(r))}
\]
elements represented by some $w \in W'$.
\end{proof}

\subsection{Some remarks on the presentation of abelian-by-$\Z$ groups}

\begin{lem}
   Let $K$ be an abelian group such that there exists an $\phi \in \Aut(K)$ and a finite subset $R \subeq K$ with $K=\left\langle\phi^i(r) \mid i \in \mathbb{Z}, r \in R \right\rangle$. Then   $K$ is a generated by  $R$ as a $\mathbb{Z}\left[t, t^{-1}\right]$-module via the action  $t (g) = t \cdot g = \phi (g)$ for all $g \in K$.  
\end{lem}

\begin{proof}
 {Suppose $K=\left\langle\phi^i(r) \mid i \in \mathbb{Z}, r \in R \right\rangle$.  Since $K$ is abelian, every $k \in K$ can be expressed as 
$$k=\sum_{r \in R} P_r(\phi)\left(r\right)=\sum_{r \in R} P_r(t)\left(r\right),$$ 
where  $P_r $ is some Laurent polynomial for each $r \in R$.}
\end{proof}

The finitely generated $\mathbb{Z}\left[t, t^{-1}\right]$-modules have been classified in \citep{modules}.  {The classification makes use of the following notation: Let $B, B^{\prime}$ and $U$ be abelian groups and $f: U \rightarrow B, g: U \rightarrow B^{\prime}$ be homomorphisms, then we consider the external direct sum $B \oplus B^{\prime}$ modulo the elements $(f(u),-g(u))$, where $u$ ranges over a set of generators for $U$. We denote the resulting group by $B \oplus_U B^{\prime}$.}

\begin{thm} \label{classification}
 Assume that $M$ is a finitely generated $\mathbb{Z}\left[t, t^{-1}\right]$-module.
 There exists a pair $U, B$ of finitely generated abelian groups and monomorphisms $f, g: U \rightarrow B$ such that $M$ is isomorphic to the infinite free product:
\begin{equation} \label{finitely_generated_module}
\cdots \oplus _U B \oplus_U B \oplus_U \cdots
\end{equation}
with identical amalgamations $B \stackrel{g}{\leftarrow} U \stackrel{f}{\rightarrow} B$. Furthermore, the symbol $t$ corresponds to the natural automorphism $\p$ that shifts every coordinate one place to the right. 

\end{thm}

 {
Given an abelian group $H$, we will denote $ rank(H)$ to be the cardinality of a maximal linearly independent subset of $H$; equivalently, $rank(H)$ is the largest $r$ such that $H$ contains an isomorphic copy of $\Z^r$.
}
\begin{notation} \label{notation_abelian}

We will often use the following notation for the finitely generated $\mathbb{Z}\left[t, t^{-1}\right]$-module described in the above theorem.
Let $B \cong \mathbb{Z} ^n\oplus \operatorname{Tor} B$ be the finitely generated abelian group with rank $n$ and let $\tau$ be the size of $\operatorname{Tor} B$. Let $\L$ and $\R$  be two isomorphic subgroups of $B$   and $ \psi: \R \rightarrow \L$ be an isomorphism. Let $r =\rank(\L) =\rank(\R)$. Then we may define a generating set for $\R$ of the form $$\{\br_1 , \ldots,  \br_r, \br_{r+1}, \ldots,  \br_{r'} \},$$ where  $\br_1 , \ldots,  \br_r$ are the only elements with infinite order in this set.  Hence, $$\L=\left\langle  \bl_1 \coloneqq \psi\left(\br_1\right), \ldots, \bl_{r'} \coloneqq  \psi\left(\br_{r'}\right)\right\rangle.$$  Let $\B = \bigoplus_{i \in \mathbb{Z}}B$. For every $i \in \Z$, we will let $B^{(i)}$,$\L^{(i)}$, $\R^{(i)}$, $\s _{i+1,i}$ be a copy of $B$, $\L$, $\R$, $\s$ respectively. Similarly, given  $x \in \Z^n$, we will let $x^{(i)}$ be the element  in $\B$ whose $i$-th coordinate is $x$ and other coordinates are $0$. Finally, define
\begin{equation} \label{normal_subgp}
 N = \left\langle \bl_j^{(i)} -  \br_j^{(i)} \mid i \in \mathbb{Z},  j \in [r']\right\rangle. 
\end{equation}
By  identifying $\L$ with $\Im(f)$ and $\R$ with $\Im(g)$,  the finitely generated $\mathbb{Z}\left[t, t^{-1}\right]$-module $M$ defined in \eqref{finitely_generated_module} can be written as $ \B / N $.  In the case when $r = n$, i.e. $\rank (\L) = \rank( B)$,  after fixing a basis for $\Z^n$,
we may  associate a matrix $Q = Q_\s \in \glnq$ with $\s$ where $Q( \tau \br_j) =  \tau\bl_j$ for every $j$, i.e.  $$Q =  \big[\tau\bl_1 \mid \ldots \mid \tau\bl_n \big]\big[\tau\br_1  \mid \ldots \mid \tau\br_n \big]  ^{-1}.$$

 
\end{notation}

\begin{lem}   \label{rank_M_is_finite_when_L_and_B_have_the_same_rank}
Using the notation defined above, we have the following equivalence:
\[\rank (\L) = n \Leftrightarrow  rank(M) < \infty.\]

\end{lem}

\begin{proof}
We will first prove $(\Leftarrow)$ via a contrapositive argument. Suppose $r <n $, let $\bl \in \Z^n$ be a vector that is linearly independent to $ \tau \bl_1, \ldots,  \tau\bl_r$. Then it is clear to see that $ \{\bl^{(i)}N \mid i \in \Z \} $ is an infinite set of linearly independent elements in the abelian group $M \iso \B / N$. 

For $(\Rightarrow)$, suppose $r=n$, we will show that   $I = \{ e_j^{(0)}N \mid  j \in [n] \} $ is a maximal linearly independent subset of $M$. Note that it is sufficient to show that for every $w \in \Z^n$ and $i \in \Z$,  a non-zero multiple of $ w^{(i)}N$ can be written as a linear combination of elements in $I$.  To see this, let  $ \b_1, \ldots, \b_n, D \in \Z$ such that $\sum _{j \in [n]} \b_j e_j = D Q^{i} w $, then we have  $ \tau \sum _{j \in [n]} \b_j e_j^{(0)}  - \tau  D w^{(i)} \in N$.
\end{proof}

\begin{lem} [$rank(M) < \infty$ and  $M$ is torsion-free] \label{M_finite_rank-M_torsion_free}
Suppose $rank(M) < \infty$ and  $M$ is torsion-free.  Then $
M \cong\left\langle Q^i\left(\mathbb{Z}^n\right) \mid i \in \mathbb{Z}\right\rangle.$

\end{lem}

\begin{proof}
Since $M =  \oplus_{i \in \mathbb{Z}}B \bigg / N$ is torsion-free, each $B$ must be torsion-free as well. For simplicity of notation, we choose a basis $\left\{y_1, \ldots, y_r\right\}$ for $ B \iso \mathbb{Z}^n$, such that  the subgroup 
 $\R $ of $B$  can be generated by $  d_1 y_1, \ldots,   d_n y_n
 $ for some $d_i \in \N$. Again, we will denote $\br_i $ for $ d_i y_i$. For $i \in \mathbb{Z}, j \in[n]$, define $$v_{i,j} = \bl_j^{(i)} -  \br_j^{(i)} =\psi_{i+1,i}\left(d_j y_j^{(i+1)}\right)-d_j y_j^{(i+1)} = d_j \sum_{k=1}^n q_{k, j} y_k^{(i)}-d_j y_j^{(i+1)},$$  then the normal subgroup  in \eqref{normal_subgp} can be written as $ N  =\left\langle v_{i,j} \mid i \in \mathbb{Z}, j \in[n]\right\rangle$.
Consider the map
$$
\begin{aligned}
\Psi: \bigoplus_{i \in \mathbb{Z}}B_{(i)} & \rightarrow\left\langle Q^i\left(\mathbb{Z}^n\right) \mid i \in \mathbb{Z}\right\rangle \\
y_j^{(i)} & \mapsto Q^i\left(e_j\right).
\end{aligned}
$$
Note this is clearly a subjective homomorphism. We will show that $\ker \Psi=N$. For ($\supseteq$) direction,  it is enough to check that the generators of $N$ are in $\ker \Psi$. For $i \in \mathbb{Z}, j \in[n]$, we have

\begin{align*}
 \Psi\left(d_j \sum_{k=1}^n q_{k, j} y_k^{(i)}-d_j y_j^{(i+1)}\right) 
=&  d_j \sum_{k=1}^n q_{k, j} \Psi{\left(y_k^{(i)}\right)}-d_j \Psi\left(y_j^{(i+1)}\right) \\
=&  d_j \sum_{k=1}^n q_{k, j} Q^i\left(e_k\right)-d_j Q^{i+1}\left(e_j\right) \\
=&  d_j Q^i\left(\sum_{k=1}^n q_{k, j}\left(e_k\right)-Q\left(e_j\right)\right) =0.
\end{align*}

Next, we will verify ($\subseteq$) direction. Let $x=\sum_{i=-L}^L \sum_{j=1}^n \alpha_{i, j} y_j^{(i)} \in \bigoplus_{i \in \mathbb{Z}} B_{(i)}$. We will denote by $q_{k, j}^{(i)}$ the $(k,j)$-th entry of $Q^i$. Suppose we have  $x  \in \operatorname{ker}(\Psi)$, then 

\begin{align*}
 0  =&\Psi\left(\sum_{i=-L}^L \sum_{j=1}^n \alpha_{i, j} y_j^{(i)}\right)  =\sum_{i =- L}^L \sum_{j=1}^n \alpha_{i, j} \Psi(y_j^{(i)})  =\sum_{i =- L}^L \sum_{j=1}^n \alpha_{i, j} Q^i\left(e_j\right) \nonumber\\
  =& \sum_{i =- L}^L \sum_{j=1}^n \alpha_{i, j} Q^{L+i}\left(e_j\right)  \quad  \text{since $Q$ is non-singular} \nonumber\\ 
  =&\sum_{i=0}^{2 L} \sum_{j=1}^n \alpha_{-L+i, j} Q^i\left(e_j\right).  
\end{align*}

Therefore, for $k \in [n]$, we have    $\sum_{i =0}^{2L} \sum_{j=1}^n \alpha_{-L+i, j} q_{k, j}^{(i)}=0.$ Since $ M=\bigoplus_{i \in \mathbb{Z}} B_{(i)} \bigg/  N$ is torsion-free, it is enough to show that $mx \in N$ for some $m$. In particular, we can assume that $d_k \vert \a_{i,j}$ for all $i,j,k$.
We claim that
$$
x=-\sum_{s=L}^{L-1} \sum_{j=1}^n\sum_{k=1}^{n}  \frac{1}{d_k} \sum_{i=0}^{L-s-1} \alpha_{s+1+ i, j} q_{k, j}^{(i)} v_{s, k}.
$$
For $s \in[-L, L-1], t \in[n]$, we will show that the coefficient for $y_t^{(s)}$ on the RHS is $\a_{s, t}$. Note that for $k \in[n]$, each $v_{s, k}$ contributes $d_k q_{t, k}$ to the coefficient of $y_t^{(s)}$  and each $v_{s-1, t}$ contributes $ - d_t$ to the coefficient of $ y_t^{(s)}$.

For $s > - L $, the coefficient for $y_t^{(s)}$ on the RHS is

\begin{align*}
&-\sum_{j=1}^n \sum_{k=1}^n \sum_{i=0}^{L-s-1} \alpha_{s+1+i, j} q_{k, j}^{(i)} q_{t, k}+\sum_{j=1}^n \sum_{i=0}^{L-s} \alpha_{s+i, j} q_{t, j}^{(i)} \\
=&-\sum_{j=1}^n \sum_{i=0}^{L-s-1} \alpha_{s+1+i, j} q_{t, j}^{(i+1)}+\sum_{j=1}^n \sum_{i=0}^{L-s} \alpha_{s+i, j} q_{t, j}^{(i)}  \\
=&-\sum_{j=1}^n \sum_{i=1}^{L-s} \alpha_{ s+i, j} q_{t, j}^{(i)}+\sum_{j=1}^n \sum_{i=0}^{L-s} \alpha_{s+i, j} q_{t, j}^{(i)}  \\
=&\sum_{j=1}^{n} \alpha_{s, j} q_{t, j}^{(0)} 
=\alpha_{s,t}.
\end{align*}

For $s=-L $, the coefficient for $y^{(-L)}$ on the RHS is

\begin{align*}
 -\sum_{j=1}^n \sum_{k=1}^n \sum_{i=0}^{2 L-1} \alpha_{-L+1+i, j}  q_{k, j}^{(i)} q_{t, k} 
=  -\sum_{i=1}^n \sum_{i=0}^{2 L-1} \alpha_{-L+1+i, j} q_{t, j}^{(i+1)}  =\\
=  -\sum_{j=1}^n \sum_{i=1}^{2 L} \alpha_{-L+i, j} q_{t, j}^{(i)} 
=\sum_{j=1}^n \alpha_{-L, j} q_{t, j} ^{(i)}  
=\alpha_{-L, j}.   
\end{align*}
This shows that $x\in N$, and hence $\ker \Psi \subseteq N$. The result 
 follows from the isomorphism theorem for groups.
\end{proof}

\subsection{Periodic points of $\p$ in $K$}
Given an abelian group $K$ and $\phi \in \Aut(K)$, we will  define the set of periodic points of $\phi$  in $K$ to be $\P_\phi = \{ a \in K \mid \exists n \in \mathbb{N}, \text{such that } \phi^{n} (a) =a \}$; { note that $\P_\phi$ is a group under addition}. For $a \in \P_\phi$, we call the smallest positive integers $m$ such that $\phi^{m} (a) =a$ the \textit{least period} of $a$. Let $G = K \rt_\p  \langle t \rangle $  be a finitely generated abelian-by-cyclic group. Note that $K$ is in particular a  $\mathbb{Z}\left[t, t^{-1}\right]$-module. Since $\Z$ is Noetherian, by Hilbert's basis theorem, $\mathbb Z[t,t^{-1}]$ is a Noetherian ring. It follows that $K$ is a Noetherian module. Recall that all of the submodules of a   Noetherian module are finitely generated.

\begin{lem} \label{bdd_period}
Let $M$ be a Noetherian module over a ring $R$. Let $\phi$ be an automorphism of $M$. Then the least period of periodic points of $\phi$ is bounded.
\end{lem}

\begin{proof}
 It is clear that $\P_\phi$ is a finitely generated submodule of $M$. Hence $\P_\phi$ can be generated by a finite set $a_1, \ldots, a_n$. Let $N$ be the lowest common multiple of the least period of $a_1, \ldots, a_n$. It is straightforward to check that every element in $\P_\phi$ has period $N$. 
\end{proof}

We can immediately deduce the following from \cref{bdd_period}:

\begin{cor}
The set of periodic points $\P_\phi$ is a finitely generated subgroup of $K$ that is invariant under $\p$. 
\end{cor}

 {\begin{proof}
    The normality of $\P_\phi$ follows from the definition.     Since $\P_\phi$ is a finitely generated $\mathbb{Z}\left[t, t^{-1}\right]$ module, there exists $k_1, \ldots k_n$ that generated  $\P_\phi$ as a module. By \cref{bdd_period}, there exists a least period $N$ of periodic points of $\phi$. Then, when viewed as a group,  
    $$\P_\phi = \langle \phi^j (k_i) \mid i \in [n], j \in \Z \rangle = \langle  \phi^j (k_i) \mid i \in [n], j \in [N] \rangle. $$
\end{proof}}

\begin{cor} \label{poly_growth}
The group $\P_\phi \rt_ \p \Z$ has polynomial growth. 
\end{cor}

\begin{proof}
Since $\P_\phi$ is finitely generated and abelian, $\P_\phi \iso \Z^n \oplus T$ for some $n$ and torsion group $T$. Fix a basis for $\Z^n$, then the map $\phi$ corresponds to a matrix $M \in \glnz$. The matrix $M$ has finite order by \cref{bdd_period}. In particular, it follows that  $\P_\phi \rt_ \p \Z$ has polynomial growth. 
\end{proof}
Our aim is to show that the growth of $\P_\p$ in $G = K \rt_\p  \langle t \rangle $ is polynomial, i.e. the function $ r \mapsto |\P_\phi \cap S^r |$, grows polynomially in $r$. When $\P_\phi$ is a direct summand of $K$ as modules, it is clear this is the case by \cref{poly_growth}.  However, in general, $\P_\p$ might not have a normal complement. For example, consider the group $\mathbb{Z}^2\rtimes_M\mathbb{Z}$ with 
$M = \begin{psmallmatrix}   1 & 1\\ 
  0 & 1
\end{psmallmatrix}$ or the group $\langle x_i (i \in \Z) \mid  [2x_i,2x_j] \rangle \rt_\p \Z$ where $\p$ maps   $x_i$ to $x_{i+1}$ for all $i\in \Z$.

We conclude this subsection with the following simple result that we will use later. 
\begin{lem} [Periodic points ``pass'' to quotients] \label{periodic_points_survive_after_quotienting}
Given $\p \in \Aut(K)$ and $N \subg K$ such that $\p(N)=N$ (or equivalently, $N \nsubg K \rt_\p  \langle t \rangle$ ), define $\bar{\p}$ to be the induced automorphism of $\phi$ on $K/N$. Then  $\P_\p N/N \subeq \P_{_{\bar{\p}} }$.
\end{lem}

\section{Distortion of the subgroups of periodic points $\P_\p$ in  $G$}
Let $S$ be a finite
generating set of $G$ and $X$ be a finite generating set of $H \subg G$. The \textit{distortion} of  $H$ in  $G$ is the function $$\Delta_G^H(r)=\max \left\{\operatorname{dist}_X(1, h) \mid h \in H, \operatorname{dist}_S(1, h) \leqslant r\right\}.$$
 
 The subgroup $H$ is called undistorted (in $G$ ) if $\Delta_G^H(r) \asymp r$. The following are some well-known facts about subgroup distortion \citep[Proposition 8.98]{ggt}.

\begin{lem} \label{distortion_basic_property}
\begin{enumerate} [(i)]
\item \label{distortion:generating_set}
If $\widetilde{X}$ and $\widetilde{S}$ are finite generating sets of $H$ and $G$, respectively, and $\widetilde{\Delta}_G^H$ is the distortion function with respect to these generating sets, then $\widetilde{\Delta}_G^H \asymp \Delta_G^H$. 
\item \label{distortion:quotient_out_finite_normal_subgroup}
If $H$ has finite index in $G$, then $\Delta_G^H(r) \asymp r$.
\item Let $K \triangleleft G$ be a finite normal subgroup and let $H \leqslant G$ be a finitely generated subgroup contains $K$; set $\bar{G}:=G / K$, $ \bar{H}:=H / K$. Then
$
\Delta_G^H \asymp \Delta_{\bar{G}}^{\bar{H}}.
$

\item \label{distortion:tower_of_subgroups}
 If $K \leqslant H \leqslant G$, then
$
\Delta_G^K \preceq \Delta_H^K \circ \Delta_G^H.
$
\item \label{distortion:abelian}
Subgroups of finitely generated abelian groups are undistorted.

\end{enumerate}
\end{lem}

\begin{lem} [Subgroup is more distorted after quotienting out a disjoint normal subgroup]\label{distortion_pass_to_certain_quotients}  
Let $G$ be a group generated by a finite symmetric set $S$ and $H$ be a subgroup generated by a finite symmetric set $X$. Let $N \nsubg G$ such that $N \cap H = \{e\}$. Then $$\dist{H}{G} \ll\dist{HN/N}{G/N}. $$

\end{lem}

\begin{proof}
Let $x \mapsto \bar{x}$ be the quotient map from  $G \rightarrow  G/N$. 
Let $x \in H$. It is clear that $|\bar{x}|_{\bar{S}} \leq|x|_S$. On the other hand, suppose $\bar{x}=\bar{x}_1 \ldots \bar{x}_l$ for some $x_1, \ldots, x_l \in X$. Since $H \cap N=\{e\}$, we must have $x=x_1 \ldots x_l$. Hence $|\bar{x}|_{\bar{X}} = |x|_{X}$. It follows from the definition of distortion that $\dist{H}{G} \ll\dist{HN/N}{G/N}. $

\end{proof}

\begin{rem}
It is well known that subgroups of finitely generated virtually nilpotent groups are all at most polynomially distorted \citep{distortion_in_nilpotent_groups}, but this is not true for finitely generated groups with exponential growth. A simple example would be the abelian-by-cyclic group $BS(1,2)= \left\langle a, t: t a t^{-1}=a^2\right\rangle$ where the cyclic subgroup $\langle a \rangle$ is exponentially distorted in the whole group.  
\end{rem}

\subsection{Distortion of periodic points in $\langle Q^i(\Z^n) \mid i \in \Z \rangle \rt_Q \Z$ where  $Q \in \glnq$ }

Given $v = (v_1, \ldots, v_n) \in \mathbb{C}^n$, we will denote $ \| v\| = \max\{|v_i| \mid 1\leq i\leq n\}$  to be the  infinity norm of $v$.   Given $M \in GL(n,\mathbb{C})$, we will denote $ \| M\|$ to be the operator norm of $M$  using the infinity norm. 
We define the \emph{generalized eigenspace} of $M$ for $\l$ as
 $$ \GE_{\l}(M) =  \ker_{\mathbb{C}^n} (  M -  \l   I)^n. $$ In the case when we have $Q \in \glnq$, we will denote $\ker_{\mathbb{Q}^n}(Q)$ to be the kernel of $Q$ over $\mathbb{Q}^n$ and $\im_{\mathbb{Q}^n}(Q)$ to be the image of $Q$ over $\mathbb{Q}^n$.  We will denote $\GE_\mathrm{ru}(Q)$ (resp.  $\GE_\mathrm{nru} $) the direct sum of the generalized eigenspaces of $Q$  relative to eigenvalues that are (resp. are not) roots of unity over $\mathbb{C}^n$. Let $k \in \N$ be such that the only eigenvalues of $Q^k$ that are roots of unity are $1$. We have $$ \mathbb{C}  ^n = \GE_\mathrm{ru} (Q)  \oplus  \GE_\mathrm{nru} (Q) =  \GE_\mathrm{1} (Q^k)  \oplus  \GE_\mathrm{nru} (Q^k). $$

By the rank-nullity theorem, we have $$n = \dim {\ker}_{\mathbb{Q}^n}(  Q^k -     I)^n + \dim \im_{\mathbb{Q}^n}(  Q^k -     I)^n.$$ Clearly, $\ker_{\mathbb{Q}^n}(  Q^k -     I)^n  \subeq \GE_\mathrm{1} (Q^k)$. On the other hand, as  $(  Q^k -     I)^n$ preserves   $\GE_\mathrm{nru} (Q^k)$, we have  $\im_{\mathbb{Q}^n}(  Q^k -     I)^n \subeq \GE_\mathrm{nru} (Q^k)$. Therefore, we can decompose $\Q^n$ as 
$$\Q^n =  \ker_{\mathbb{Q}^n}(  Q^k -     I)^n \oplus   \im_{\mathbb{Q}^n}(  Q^k -     I)^n.$$ Let $v_1, \ldots, v_n', \ldots,v_n \in \Q^n$, such that $$\ker_{\mathbb{Q}^n}(  Q^k -     I)^n = \operatorname{span}\left\{v_1, \ldots, v_{n'}\right\} \text{ and }  \im_{\mathbb{Q}^n}(  Q^k -     I)^n = \operatorname{span}\left\{v_{n'+1}, \ldots, v_{n}\right\}.$$ Then  there exists $ L \in \N$, such that $ \frac{1}{L}\left\langle v_1, \ldots, v_n\right\rangle \supseteq \mathbb{Z}^n.$  Let $$Y = \frac{1}{L} \left\langle Q ^ { i } \left(v_j\right) | i \in \mathbb{Z}, j \in \{n'+1, \ldots n\}\right\rangle,$$ we have 
$$
K \subseteq \frac{1}{L} \left\langle Q ^ { i } \left(v_j\right) | i \in \mathbb{Z}, j \in  [n']\right\rangle \oplus Y  
.$$ Given $v \in K$, we will define $\e(v)$ to be the projection of $v$ on  $ \frac{1}{L} \left\langle Q ^ { i } \left(v_j\right) | i \in \mathbb{Z}, j \in  [n']\right\rangle$, i.e. $\e(v)$ is a vector such that $v -  \e(v) \in Y$.

\begin{lem}   \label{eigenvalue_and_growth_ratio}
Let $M \in GL(n,\mathbb{C})$ be a matrix whose eigenvalues all have absolute value $1$. Then  $\|M^r\| =\mathcal{O}(r^{n}).$ 

\end{lem}

\begin{proof}

Let $J= P^{-1}MP$ be the Jordan normal form of $M$ for some change of basis matrix $P$. By looking at the $r$-th power of $J$ ({see \citep[Chapter 3.2.5]{matrix} for a detailed computation of power of a Jordan block}), we can deduce that the absolute value of every entry of $J$ is bounded by $ \mathcal{O}(r^{n})$. 
Therefore, $$\|M^r\| = \|PJ^n P\i\| = \|P\| \|J^n\| \|P\i\|  =  \mathcal{O}(r^{n}). $$

\end{proof}

We will denote $Q'$ to be the restriction of $Q$ on $\ker_{\mathbb{Q}^n}(  Q^k -     I)^n $.



\begin{lem}  \label{lem:bounded_on_length}
Let  $w $ be an element in $ S^r$ with zero $t$-exponent sum (i.e. $w \in K \cap S^r$ ), then $\norm {\e (w)} = \mathcal{O}(r^{n+2})$. 
\end{lem}

\begin{proof}

By \cref{contains_all_elements}, we can write $$w=   t^{-b} u_0 t u_1 t   \ldots u_{l-1} t u_{d}  t^{-c},  $$ with $b, c, d \geq 0$ and $d=b+c$. Since each $u_i \in B_R(r)$,  $\norm {u_i} =\O\left(r\right) $. Therefore, according to  \cref{eigenvalue_and_growth_ratio},
\begin{align*}
\norm{ \e(t^{-b} u_0 t u_1 t   \ldots u_{l-1} t u_{d}  t^{-c} ) }= \norm{t^{-b}\e(u_0) t \e(u_1) t \ldots t\e(u_d) t^{-c} } =  \norm{\sum_{i=-b}^c Q'^i(\e(u_i)) } \\
 \leq  \sum_{i=-b}^c\ \norm{ Q'{^i}} \|\e(u_i)\|
  \leq  d  \O(d^n)\O(r)  =  \mathcal{O}(r^{n+2}).
\end{align*}

\end{proof}

\begin{cor} \label{rational_matrix_case}
The set of periodic points $\P_Q$ is polynomially distorted in $G =\langle Q^i(\Z^n) \mid i \in \Z \rangle \rt_Q  \langle t \rangle$.
\end{cor}

\begin{proof}

Let $X$ be the standard generating set for the finitely generated abelian group $\P_Q$, and let $ R $ be a finite symmetric subset of $\langle Q^i(\Z^n) \mid i \in \Z \rangle$ such that  $S =  R \cup \{t^{\pm 1}\} $ is a generating set of $G$. Using these generating sets, it follows from \cref{lem:bounded_on_length} that $\dist{\P_Q}{G} = \mathcal{O}(n r^{n+2}) =  \mathcal{O}( r^{n+2}) $. 

\end{proof}






\subsection{Generalising to other finitely generated abelian-by-$\Z$ group}

Our goal is to prove that $\P_\p$ is at most polynomially distorted in the finitely generated abelian-by-cyclic group $G =  K \rt_\p \langle t \rangle$. The following lemma will help us show that we can reduce the problem to the case where $K$ has finite rank. We will view $K$ as a  finitely generated  $\mathbb{Z}\left[t, t^{-1}\right]$-module, and use the notation introduced after \cref{classification} for $K$, i.e. we will identify $K$ as $ \B /N $.

\begin{lem} \label{P_N_prime_trivial_intersection}
Let $\bl_{r+1}, \ldots, \bl_{n}, \br_{r+1}, \ldots, \br_{n}$ be vectors in $ \Z^n$ such that  we have $$\{\tau \bl_1, \ldots \tau\bl_r,\bl_{r+1}, \ldots, \bl_{n}\} \text{ and } \{ \tau \br_1, \ldots \tau\br_r, \br_{r+1}, \ldots, \br_{n} \}$$ being two linear independent subsets of   $\Z^n$. Consider the normal subgroup $$N'  = \left\langle \bl_j^{(i)} -  \br_j^{(i)} \mid i \in \mathbb{Z},  r+1 \leq j \leq n \right\rangle$$ of $G$, then $\P_\p \cap N'N/N  = \{\id\}$

\end{lem}

\begin{proof}
Let $P$ be the subgroup of $\B$ that contains $N$ such that $\P_\p = P/N$. We will show that $P \cap N' = \{\id\}$.  Let $w \in N^{\prime}$ and $n \geq 1$ be such that $ \phi^n (w)-  w  \in N$, i.e. $wN \in \P_\p$. Suppose that $w \neq \id$. Let $i$ be the smallest integer such that the $i$-th component  $w_i$ of $w$ is non-trivial. Let $p \coloneqq \phi^n (w)-  w $,   then $i$ is also the smallest integer such that the $i$-th component  $p_i$ of $p$ is non-trivial. Furthermore, since $n \geq 1$, we have $p_i=-w_i$.  Since $w \in N^{\prime}$, $w_i \in \operatorname{span}_\Z\left\{\bl_{r+1}, \ldots, \bl _n\right\}$. On the other hand, as $p \in N$, $p_i \in \operatorname{span}_{\mathbb{Z}}\left\{\bl_{1}, \ldots, \bl_{r} \right\}$. We must have $w_i=p_i=0$, which is a contradiction. Therefore, we indeed have $P \cap N' = \{\id\}$.
\end{proof}

\begin{cor} \label{rere}
Let $G = K \rt_\p  \langle t \rangle $  be a finitely generated abelian-by-cyclic group. We have that $\dist{\P_\p }{G}$ grows polynomially in $r$. 

\end{cor}

\begin{proof}
{
Given a group $K$ and   $N \subg K$ such that $\p(N)=N$, every automorphism $\p$ of $K$ naturally induces an automorphism  $\bar{\phi}$ on $K/N$. For conciseness, we might not define $\bar{\p}$ explicitly when it is clear from the context in this proof. 
}

We will show that we can reduce the general case to the special case where $K$  has finite rank and is torsion-free, which was proved in \cref{rational_matrix_case}.

We first claim that we can reduce our problem to the case where $\rank(K) < \infty$. Suppose  $rank(K) = \infty$. Let $N'$ be a normal subgroup of $\B$ with the form described in \cref{P_N_prime_trivial_intersection}, and let $H = NN'/N$. \cref{distortion_pass_to_certain_quotients} tells us that $$\dist{\P_\p}{G} \ll \dist{\P_\p H/H} { (K/H) \rt_{\bar{\p}}\Z }.$$ According to \cref{periodic_points_survive_after_quotienting}, $\P_\p H/H \subg \P_{\bar{\p}}$. Since $\P_{\bar{\p}}$ is finitely generated and abelian,  \cref{distortion_basic_property} \eqref{distortion:tower_of_subgroups} and \eqref{distortion:abelian} tells us that $$\dist{\P_\p H/H}{ (K/H) \rt_ {\bar{\p}} \Z } \ll \dist{\P_{\bar{\p}}}{ (K/H) \rt_ {\bar{\p}} \Z }$$ which proves our claim.

We will next show that we can also reduce to the case where $K$ is torsion-free. Since $\P_\p$ is finitely generated abelian,   $\Tor(\P_\p)$ is a finite and characteristic subgroup of $\P_\p$. By  \cref{distortion_basic_property} \eqref{distortion:quotient_out_finite_normal_subgroup}, $$\dist{\P_\p}{G} \asymp  \dist{\P_\p/ \Tor(\P_\p) } { (K / \Tor(\P_\p)) \rt_{\bar{\p}}\Z }.$$ From this, we can quotient out torsion elements of $K$, it is clear that $$\frac{\P_\p \Tor(K)/ \Tor(\P_\p)  }{\Tor(K)/ \Tor(\P_\p) } \iso  \frac{\P_\p \Tor(K) }{\Tor(K) } \text { \quad and \quad  } \frac{ K / \Tor(\P_\p)  }  {\Tor(K)/ \Tor(\P_\p) }  \iso  \frac{ K }  {\Tor(K) }.$$ Using a similar idea as above,  we have 
$$\dist{\P_\p/ \Tor(\P_\p) } { (K / \Tor(\P_\p)) \rt_{\bar{\p}}\Z } \ll \dist{\P_\p \Tor(K) / \Tor(K)}{K / \Tor(K )\rt_{\bar{\bar{\p}}}\Z} \ll \dist{    \P_ {\bar{\bar{\p}}}    }{K / \Tor(K )\rt_{\bar{\bar{\p}}}\Z},$$
where the first bound follows from \cref{distortion_pass_to_certain_quotients},    and the second follows from  \cref{periodic_points_survive_after_quotienting},  \cref{distortion_basic_property} \eqref{distortion:tower_of_subgroups} and \eqref{distortion:abelian}. 
Finally, the rank of a quotient of $K$ is at most the rank of $K$, which completes our reduction process and proves our result.  
\end{proof}

Let $S$ be a finite generating set of $G$ and $X$ be a finite generating set of $H \subg G$. The \textit{relative growth} of  $H$ in  $G$ is the function $ r \mapsto |H \cap S^r |$.  The relative growth of subgroups was studied by Osin in \citep{osin}.  The following simple lemma links the distortion and the relative growth of a subgroup with polynomial growth.

\begin{lem} [Polynomial distortion implies polynomial relative growth for subgroups with polynomial growth] \label{relative_growth_of_poly_subgroup}
Let $G$ be a group generated by a finite symmetric set $S$. Let $H$  be a subgroup generated by a finite symmetric set $X$. Suppose that $|X^r| = \O(r^L) $  for some $L \in \N$ and $\dist{H}{G} = \O (r^{d})$. Then $|H \cap S^r | = \O (r^{d|X|})$.

\end{lem}

\begin{proof}

It follows from the definition of distortion that $H \cap S^{r} \subseteq X^{\dist{H}{G}}$. Therefore, 
$$
\left|H \cap S^r\right| \leq \left|X^{\dist{H}{G}}\right|=\O\left(\dist{H}{G}^{L}\right)= \O (r^{dL}).
$$

\end{proof}
 For more results about how distortion is linked to relative growth, we refer the readers to \citep{relative}.

\begin{cor} \label{growth_of_p_in_G}
    The relative growth of  $P_\p$ in $G$ is polynomial.
\end{cor}

\section{Conjugacy ratios}
In Ciobanu, Cox and Martino's paper \citep{conjugacy_ratio_Laura_Charles}, they also investigated the spherical conjugacy ratio, where the counting is done in the sphere rather than the ball. We will call a group element $g$ a  \textit{minimal length conjugacy representative} if $|g|=\min \{|h| \mid h \sim g\} $. Denote $S'(r) \subeq S(r)$ to be the set of minimal length conjugacy representatives on $S(r)$, in particular, $c(S'(r)) = c(S^r) - c(S^{r-1})$.  The spherical conjugacy ratio is defined to be  $ \underset{r \rightarrow \infty}{\limsup}  \frac{c(S'(r))}{\left|S(r)\right|}$.  By the Stolz-Cesàro theorem, it is straightforward to see that if $ \underset{r \rightarrow \infty}{\lim}  \frac{c(S'(r))}{\left|S(r)\right|} =0$, then   $\underset{r \rightarrow \infty}{\lim} \frac{ c(S^r)
}{\left|S^r\right|} =0$.  We will in fact give a quantitative bound for the spherical conjugacy ratio function and show that the same bound also applies to the standard conjugacy ratio. We will introduce one more asymptotic notation: let $f$ and $g$ be two real-valued functions defined on some unbounded subset of real numbers. We write $f(x)=o(g(x))$ if for every positive constant $\varepsilon$ there exists a constant $x_0$ such that
\[
|f(x)| \leq \varepsilon g(x) \quad \text { for all } x \geq x_0.
\]
Roughly speaking, $f(x)=o(g(x))$ if  $f(x)$ grows much slower than $g(x)$.


Intuitively, we want to show that most elements on the sphere are conjugate to a lot of other elements on the sphere; equivalently, the set of elements on the sphere that are conjugate to very few elements on the sphere is small.   {Let $f(r)$ be a function that converges to $ \infty$}. Given $r \in \N$,  denote $$ F(r) = F_{f, S}(r) = \left\{  g  \in S'(r) \bigg|  |[g]_G \cap S(r)|  \leq f (r)  \right\}. $$ We formulate our idea in the following lemma, and then show that it applies to our setup in the subsequent theorem.


\begin{lem} \label{F_r_is_small_implies_spherical_conjugacy_ratio_is_0}
 Suppose we have $c(F(r)) = \mathcal {o} (|S(r)|)$ and $f(r) = \mathcal {O} (|S(r)|)$. Then $\frac{c(S'(r))}{\left|S(r)\right|}= \mathcal{O} \left({\frac{1}{f(r)}}\right) $.
\end{lem}

\begin{proof}

Note that $c(S'(r)) \leqslant|S'(r)| \leqslant|S(r)|$. Also,  it follows that $F(r) \leq c(F(r)) f(r)  = \mathcal {O} (|S(r)|)$ by the definition of $F(r)$. 
Therefore,
$$
 \frac{c(S'(r))}{|S(r)|} = \O \left ( \frac{c(S'(r))-c(F(r))}{\left|S(r)\right| -|F(r)|} \right ) = \O \left ( \frac{c(S'(r))-c(F(r))}{\left|S'(r)\right| -|F(r)|} \right ) =\O \left ( \frac{1}{f(r)} \right ) .$$

\end{proof}

Our next goal is to give an upper bound for $c(F(r))$ for the abelian-by-$\Z$ group with respect to a generating set defined in \cref{main_thm}. It is clear from \cref{contain_all_conjuacy_geodesic} that $\mathcal{C}_0 \cup \mathcal{C}_+$ contains all the representatives of the conjugacy classes of $G$ with non-negative $t$-exponent sum.

\begin{thm} [On $S(r)$, set of elements that can produce at most $f(r)$ distinct elements has size bounded by $ \mathcal {O} (r^{f (r)})$]\label{bound_on_Fr}
Let $G = K \rt_\p  \langle t \rangle $  be a finitely generated abelian-by-cyclic group.  Let $f(r)$ be a function that converges to $ \infty$ in $r$. Let $ R $ be a finite symmetric subset of $K$ such that $S = \{  (r,1),(0,t^{\pm 1}) \mid r \in R \}$ is a generating set of $G$. Then $c(F(r)) = r^{\mathcal{O}(f(r))}$. 
\end{thm}

\begin{proof}

Given a word $w$, we will denote $\c{w}$ as the number of distinct group elements in $\cpc{w}$. We will first look at the minimal length conjugacy representative with zero t-exponent sum. It is clear that, for $x \in K \setminus \P_\phi$, and a set of integers $I$, the set $\{ t^{-i} x t ^i \mid i \in I \}$ has the same cardinality as $I$. 
In particular, given $g\in  (K \cap S(r)) \setminus ( U_f (r)   \cup \P_\phi) $, and $w \in \mathcal{C}_0$  a word that conjugates to $g$, we have $\c{w} \geq f (r)$. It follows from \cref{most_elts_have_many_t} and \cref{growth_of_p_in_G} that     $$  c \left ( \left\{g \in S'(r) \cap K   \biggm  | |[g]_G \cap S (r)| < f (r) \right\} \right)$$  has size $r^{\mathcal{O}(f(r))}$.

We will next look at the conjugacy geodesics with non-zero $t$-exponent sum. Note that taking inverses preserves word length and conjugacy class, it follows that  $\mathcal{C}_+ \i$ contains the conjugacy class representatives with negative $t$-exponent sum. Therefore, it is sufficient to only examine the conjugacy geodesics in $\mathcal{C}_{+}$.  Our aim is to bound the number of minimal length conjugacy representatives in $ S(r) \setminus K$ that can be expressed as a word $w$ in $\mathcal{C}_{+}$  with  $\c{w}  < f (r)$.

 We will begin by taking a look at a way of expressing a word in $  \mathcal{C}_{d+1}$, where $d \geq 0$, with one of its cyclic permutations. Let $u = a_0 t a_1 t  a_2 t  \ldots a_{d} t \in \mathcal{C}_{d+1}$. For $i \in[2 d+1]$, define $a_i = a_j$ whenever $i=_{d+1} j$, i.e. addition is defined $\bmod (d+1)$ for the indices. Define $$u_{(i)} =  (a_i t a_{i+1} t    \ldots a_{d} t a_{0} t \ldots a_{i-1} t) t^{-(d+1)},$$ i.e. $u_{(i)}$ is the $K$ component of cyclic permutation of $u$ that starts with $a_i$.  Given $\delta \in\{0, \ldots, d-1\}$, we have  
$$
\p^{\delta+1}  u_{(\d+1)}= \p^{\delta+1} \sum_{j=0}^{d  } \p^j a_{j+\delta+1}=\sum_{j=0}^{d} \p^{j+\delta+1} a_{j+\delta+1}=\sum_{j=\delta+1}^{d+\delta+1} \p^j a_j.
$$
Hence, 
\begin{align*}
u_{(0)} = \sum_{j=0}^d \p^j a_j  
&=\p^{\d+1} u_{(\d+1)}+\sum_{j=0}^\delta \p^j a_j-\sum_{j=d+1}^{d + \d+1} \p^j a_j  \\
&=\p^{\d+1} u_{(\d+1)}+\sum_{j=0}^\delta \p^j a_j-\sum_{i=0}^\delta \p^{j+d+1} a_{j+d+1} \\
&=\p^{\d+1} u_{(\d+1)}+\sum_{j=0}^\delta \p^j a_j- \p^{d+1} \sum_{j=0}^\delta \p^j a_j \\
&=\p^{\delta+1} u_{(\d+1)}+\left(\id-\p^{d+1}\right) \sum_{j=0}^\delta \p^j a_j.  \numberthis \label{express_u0}
\end{align*}

Now, suppose that $d \geq f (r)$, let $w \in \mathcal{C}_{d+1} $ be a conjugacy geodesic of length $r$. Suppose $\c{w} <f (r)$. Then, there exists $u \in \cpc{w} \cap \mathcal{C}_{d+1} $ and  $ \delta \in[f (r)-1]$ such that $u_{(0)}= u_{(\delta+1)}$. Using the above notation for $u$, we can deduce from \eqref{express_u0} that

\begin{equation} \label{I-phi_bound}
 \left(\id-\p^{\d+1}\right) u_{(0)}=\left(\id-\p^{d+1}\right) \sum_{j=0}^\delta \p^j a_j. 
\end{equation}

Consider the RHS of \eqref{I-phi_bound}, we  have $d \in[r]$, $\d \in[f (r)]$,  and $a_0, \ldots, a_\d $ represents  elements in $ B_R(r)$ where $\left|B_R(r)\right| \in \O( r^{|R|})$. Next, given such $d$, $\d$, and $a_i$, suppose there exist $h, h' \in S^{2r} \cap K$ such that both $u_{(0)} = h$ and  $u_{(0)} = h'$ are solutions to \eqref{I-phi_bound}. We must have    $ h - h' \in \operatorname{ker}(\id-\p^{\d+1})$. In fact,  the set of such solutions is contained in  $\left\{k \in K    \mid k-h \in \P_\p \cap S^{4r}      \right\} $. According to \cref{growth_of_p_in_G}, there exists $P(r)$, a polynomial in $r$,  such that $$\left| \left\{k \in K    \mid k-h \in \P_\p \cap S^{4r}      \right\}\right|  = \O(P(r)).$$ Therefore, $$ c \left(  \left\{g \in S'(r) \setminus   (K \cup U_f (r))  \biggm  | |[g]_G \cap S (r)| < f (r) \right \} \right)$$  is bounded by $$r  \cdot f (r) \cdot \O (r^{|R| })^{f (r)} \cdot \O(P(r))  =  r^{\mathcal{O}(f(r))}.$$ Again, by  \cref{most_elts_have_many_t}, we have $$ c\left  (  \left \{g \in S'(r) \setminus   K  \biggm  | |[g]_G \cap S (r)| < f (r) \right\}   \right ) =r^{\mathcal{O}(f(r))}.$$

\end{proof}

{We are now ready to prove \cref{main_thm}. We will begin by showing that the same quantitative bound applies to the spherical conjugacy ratio. 
}

\begin{thm} \label{main_thm_stronger}
Let $G = K \rt_\p  \langle t \rangle $  be a finitely generated abelian-by-cyclic group with exponential growth.  Let $ R  $ be a finite symmetric subset of $K$ such that $S = \{  (r,1),(0,t^{\pm 1}) \mid r \in R \}$ is a generating set of $G$. Then $$ \frac{c(S'(r))}{\left|S(r)\right|} = \O\left(\frac{\log r}{r}\right).$$ 
\end{thm}

\begin{proof}
Since $S^r$ grows exponentially in $r$, by the isoperimetric inequality, we conclude that $S(r)$ also grows exponentially in $r$.  According to \cref{bound_on_Fr}, any $f(r) = \littleo\left(\frac{r}{\log r}\right) $ satisfies the assumption of \cref{F_r_is_small_implies_spherical_conjugacy_ratio_is_0}, i.e. 
 $$\frac{c(S'(r))}{|S(r)|} = \O \left ( \frac{1}{f(r)} \right ), $$ from which our result follows.  \end{proof}


\begin{proof} [Proof of \cref{main_thm}]
 For simplicity of notation, denote $c_r = c(S'(r))$, $ b_r = \left|S(r)\right|$, $C_r =  c(S^r) = c_1+c_2+\cdots+c_r$ and $B_r = \left|S^r\right| = b_1+b_2+\cdots+b_r$. By \cref{main_thm_stronger}  and the definition of big-$\O$ notation,  there exists a positive real number $M$ such that $ \frac{ c_r}{ b_r} \leq  \frac{M\log r}{r} $ for all $r$. Denote $H(r) = \frac{M\log r}{r} $.   Then for any $k \in \N$, we have

\begin{equation*}
      \frac{ C_r}{ B_r} \leq \frac{c_1+\cdots+c_k + H(k+1)b_{k+1} + \cdots + H(r)b_{r} }{B_r}
    \leq \frac{C_k + H(k) ( B_r - B_k)}{B_r}
    \leq H(k) + \frac{C_k}{B_r}.
\end{equation*}

Since $B_r$ grows exponentially in $r$ and  $|C_r| \leq |B_r|$ for all $r$, there exists $L$ such that  $\frac{C_{r/L}}{B_r} =  \O (\a^{-r}) $ for some $\a >1$. Taking $k =  \frac{r}{L}$,  it follows that  $$\frac{ C_r}{ B_r} \leq   H \left(\frac{r}{L} \right) + \O (\a^{-r}) = \O\left(\frac{\log r}{r}\right). $$


\end{proof}

 {
\begin{proof}[Proof of \cref{main_thm_cor}]  
According to the Milnor-Wolf Theorem on the growth of solvable groups, an abelian-by-cyclic group has either polynomial or exponential growth. The case where $G$ has polynomial growth was proven by Ciobanu, Cox and Martino in \citep[Theorem 3.7]{conjugacy_ratio_Laura_Charles} (as well as by Tointon\citep[Corollary 8.2 and Proposition 8.5]{commuting_tointon}). The case where  $G$ has exponential growth follows from \cref{main_thm}.
\end{proof}}

 Our result in particular answered a question raised by Cox in \citep[Question 1]{conjugacy_ratio_Laura_Charles} for the groups and generating sets defined in \cref{main_thm_stronger}.




    

\section{Conjugacy ratio with respect to one-sided Følner sequence }
In this section, we prove \cref{main_thm_2}. Let $k \geq 2$, we will consider the Baumslag-Solitar group $$G = BS(1, k)= \left\langle a, t \mid t a t^{-1}=a^k\right\rangle \iso  \mathbb{Z}   \left[\frac{1}{k}\right] \rtimes_\p \langle t \rangle$$
where $\phi \in \Aut(\mathbb{Z}   \left[\frac{1}{k}\right])$ is defined by $\phi (n)= tnt\i = kn$.

Recall that, given $a, b \in \mathbb{Z}$ and $d \in \mathbb{Z} \setminus \{0\}$ such that $d$ does not divide $b$, we say that $\frac{a}{b} \equiv_d k$ if $kb \equiv_d a$.

\begin{lem} \label{when_are_a_b_in_same_class}
Given $a t^n$ and $b t^n$ in $G$ with $a$, $b$ being two positive integers and $n \in \N$, we have 
\begin{equation} \label{conjugate}
   a t^n \sim b t^n \Leftrightarrow \exists m \in \mathbb{Z}, \text{ such that }  k^m a  \equiv b \quad\left(\bmod \; k^n-1\right).
\end{equation}

\end{lem}

\begin{proof}
Let $a $ and $ b$ be two positive integers. An easy computation shows that
$$\left[a t^n \right]=\left\{  \bigg(\left(k^n-1\right) c+k^m   a,\quad t^n \bigg) \biggm | c \in \mathbb{Z}\left[\frac{1}{k}\right], m \in \mathbb{Z}\right\}.$$

The $(\Leftarrow)$ direction of \eqref{conjugate} then follows immediately. For $(\Rightarrow)$, suppose $at^n \sim bt^n$. Then $b = \left(k^n-1\right) c+k^ma$    for some $ c \in \mathbb{Z}\left[\frac{1}{k}\right]$ and  $m \in \mathbb{Z}$. Suppose $m \geq 0$, then we have   $(k^n-1) c = b -k^m a \in \Z$. Note that for any $r \in \N$, we have $\gcd  (k^{n}-1, k^r) = 1$. It follows that $ c$ must be in $ \Z$. If $m \leq 0$, then we have $(k^n-1) k^{-m} c = k^{-m} b - a \in \Z$. By the same reasoning, we must have $k^{-m} c \in Z$. This concludes the proof.

\end{proof}

\begin{lem} \label{lemma_fintely_many_n_where_the_quotient_power_of_2}

Let $a $ and $ b$ be two positive integers such that  $\frac{b}{a} \notin \{k^i \mid i \in \mathbb{Z} \}$, i.e.
$\frac{b}{a}$ is not a power of $k$.   Then there are only finitely many $n \in \N$ such that 
\begin{equation} \label{fintely_many_n_where_the_quotient_power_of_2}
\text{ there exists } j=j_n \in \Z  \text{ with }  k^j a \equiv b \pmod {k^n-1}.
\end{equation}

\end{lem}

\begin{proof}

Let $a $ and $ b$ be two positive integers. Let $n \in \N$. Suppose that there exists $j = j_n \in \mathbb{Z}$ such that $k^j a \equiv b \pmod{k^n - 1}$. Note that $k^n \equiv 1\left(\bmod \;  k^n-1\right)$, so we can take $j = j_n \in\{0,1, \ldots, n-1\}$. Since $k^n-1 \mid k^j a -b$, we have $k^n-1   \biggm| k^{n-j}\left(k^j a-b\right)$, note that 
$$
k^{n-j}\left(k^j a-b\right) =k^na  -k^{n-j} b= \left(k^n-1\right)a-k^{n-j} b+a.
$$
It follows that $k^n-1 \mid k^{n-j} b-a $. Since $a $ and $ b$ are two positive integers, we have $ k^j a-b\geq -b $ and $k^{n-j} b-a \geq -a$. Suppose $n $ is big enough such that $k^n-1 > \max\{a,b\}$, then no elements in $\{-\max\{a,b\}, \ldots, -2, -1\}$ can be a muliple of $k^n-1$, so we must have $ k^j a-b\geq 0 $ and $k^{n-j} b-a \geq 0$.

Suppose $\frac{b}{a}$ is not a power of $k$. It follows that both $k^j a-b>0$ and $k^{n-j} b-a>0$. Again, using that $k^n-1 \mid k^j a -b$ and  $k^n-1 \mid k^{n-j} b-a $, we have
$$
k^j a  -b \geqslant k^n-1 \text { and   } k^{n -j}b-a \geqslant k^n-1.
$$

By rearranging these two inequalities, we get
\begin{equation} \label{inequalities}
\frac{k^n-1+b}{a} \leq k^j \leq \frac{k^n b}{k^n-1+a}; 
\end{equation}

Since the left-hand side grows faster than the right-hand side, \eqref{inequalities} can only be satisfied by finitely many values of $n$.

\end{proof}

\begin{lem} \label{same_t_exponent}
Let $n \in \mathbb{Z}$, and let $A =\left\{\left(a_1, t^n\right), \ldots,\left(a_r , t^n\right)\right\} \in \mathbb{Z}[\frac{1}{k}] \rtimes_\phi \mathbb{Z}$   be a set of $r$ elements with $t$-exponent sum equal to $n$. Then, for big enough $N_1, N_2$ and $L$ we have 
$$
c\left(t^{N_1} L t^ {N_2}  A\right)=|A|
$$

\end{lem}

\begin{proof}
For big enough $N_2 \in \N$ we have $k^{N_2} a_i = t^{N_2} a_i t^{-N_2} \in \mathbb{Z}$ for   all $i$. Taking a positive integer $L  \geq 2\max \{ |k^{N_2} a_i| \mid i \in [r]\}$, we can see that $A^{\prime}=\left\{L+k^{N_2} a_i  \mid i \in[r]\right\}$ is a set of positive integers such that it has no pair of elements with the quotient being a power of $k$, i.e. for any $a,b \in A'$ and $ i \in \mathbb{Z}$,  $\frac{b}{a} \neq k^i$. Finally, for big enough $N_1 > -n$,  we can assume no pair of elements in $A^{\prime}$  satisfies \eqref{fintely_many_n_where_the_quotient_power_of_2} modulo $k^{N_1 + N_2+n}-1$.
By \cref{when_are_a_b_in_same_class} and \cref{lemma_fintely_many_n_where_the_quotient_power_of_2}, every element in $$ 
\left\{\left(L+k^{N_2} a_i, t^{N_1 + N_2+n}\right) \mid i \in[r]\right\}
$$
is in a distinct conjugacy class. Furthermore, we have
$$t^{N_1} L t^{N_2} (a_i, t^n)=t^{N_1}\left(L+k^{N_2} a_i\right) t^{n+N_2} \sim\left(L+k^{N_2} a_i\right) t^{ N_1 + N_2+n  }.$$
Therefore, every element from $t^{N_1} L t^{N_2} A$ is in a distinct conjugacy class.

\end{proof}

\begin{cor} \label{same_size_set_class}
Given any finite set $A \subset G$, there exists $g \in G$, such that $$ c\left(g  A\right)=|A|.
$$
\end{cor}

\begin{proof}
This follows from \cref{same_t_exponent} and the fact that two conjugate elements have equal $t$-exponent sum.
\end{proof}

\begin{proof} [Proof of \cref{main_thm_2}]
Let $\left(F_n\right)_{n=1}^{\infty}$ be a right  Følner sequence of $BS(1, k)$. According to \cref{same_size_set_class}, for every $n$, there exists $g_n \in G$ such that $ c\left(g_n  F_n\right)=|F_n| = |g_n F_n|$. It follows from the definition \eqref{right_folner} that $\left( g_n F_n\right)_{n=1}^{\infty}$ is also a right Følner sequence of $BS(1, k)$. Furthermore,  by the choice of $g_n$, we have $\frac{  c(g_n F_n)  }{\left| g_n F_n\right|} = 1$ for every $n$, which proves our result.

\end{proof}

\footnotesize{
\bibliographystyle{abbrv}
\bibliography{Biblio}
}

\end{document}